\documentclass[12pt, fleqn]{article}

\usepackage{setspace}
\usepackage[ruled,linesnumbered]{algorithm2e}
\usepackage[margin=2.5cm]{geometry}
\usepackage{amsthm, amssymb, amstext}
\usepackage[fleqn]{amsmath}
\usepackage{latexsym}
\usepackage[dvips]{graphicx}
\usepackage{relsize}
\usepackage{hyperref}
\usepackage{mathtools}
\usepackage{enumerate} 
\usepackage{tikz}
\usetikzlibrary{arrows,shapes,calc}
\usepackage[ruled]{algorithm2e}

\usepackage{todonotes}
\usepackage{comment}

\def\l{\lambda}
\def\A{\mathcal{A}}
\def\B{\mathcal{B}}
\def\C{\mathcal{C}}

\def\F{\mathcal{F}}

\def\S{\mathcal{S}}

\def\H{\mathcal{H}}
\def\G{\mathcal{G}}
\def\RW{A}

\def\RWp{\RW$^\prime$}
\def\NW{D$^\prime$}

\let\oldmarginpar\marginpar
\renewcommand\marginpar[1]{\-\oldmarginpar[\raggedleft\footnotesize #1]%
{\raggedright\footnotesize #1}}

\newtheorem{theorem}{Theorem}[section]
\newtheorem{lemma}[theorem]{Lemma}

\newtheorem{corollary}[theorem]{Corollary}
\newtheorem{conjecture}{Conjecture}[section]

\newtheorem{problem}{Problem}

\theoremstyle{remark}
\newtheorem{claim}{Claim}

\theoremstyle{definition}
\newtheorem{definition}{Definition}[section]
\newtheorem*{remark*}{Remark}

\usepackage{amsthm}

\title{Enclosings of Decompositions of Complete Multigraphs in $2$-Edge-Connected $r$-Factorizations}

\author{
	John Asplund\thanks{Department of Technology and Mathematics,
Dalton State College,
Dalton, Georgia 30720, USA \newline email: \texttt{jasplund@daltonstate.edu}}, \quad
    Pierre Charbit\thanks{IRIF \& Universit\'e Paris Diderot, Inria Project-team GANG \newline email: \texttt{charbit@irif.fr}}, \quad
	Carl Feghali\thanks{University of Bergen, Bergen, Norway \newline email: \texttt{carl.feghali@uib.no} } 
}

\date{}

\begin{document}
\maketitle

\begin{abstract}
A decomposition of a multigraph $G$ is a partition of its edges into subgraphs $G(1), \ldots , G(k)$. It is called an $r$-factorization if every $G(i)$ is $r$-regular and spanning. If $G$ is a subgraph of $H$, a decomposition of $G$ is said to be enclosed in a decomposition of $H$ if, for every $1 \leq i \leq k$, $G(i)$ is a subgraph of $H(i)$.

Feghali and Johnson gave necessary and sufficient conditions for a given decomposition of $\lambda K_n$ to be enclosed in some $2$-edge-connected $r$-factorization of $\mu K_{m}$ for some range of values for the parameters $n$, $m$, $\lambda$, $\mu$, $r$: $r=2$, $\mu>\lambda$ and either $m \geq 2n-1$, or  $m=2n-2$ and $\mu = 2$ and $\lambda=1$,  or $n=3$ and $m=4$. 
We generalize their result to every $r \geq 2$ and $m \geq 2n  - 2$. We also give some sufficient conditions for enclosing a given decomposition of $\lambda K_n$ in some $2$-edge-connected $r$-factorization of $\mu K_{m}$ for every $r \geq 3$ and $m = (2 - C)n$, where $C$ is a constant that depends only on $r$, $\lambda$ and~$\mu$. 
\end{abstract}


\section{Introduction}

In this paper, graphs are undirected and may contain multiple edges but no loops. The set of vertices and edges of a graph $G$ are denoted by~$V(G)$ and~$E(G)$ respectively. A \emph{decomposition} of $G$ into $k$ colors is a collection $\G =  \{\G(1), \G(2), \dots, \G(k)\}$ of spanning subgraphs of $G$, called \emph{color classes}, whose edge sets form a partition of $E(G)$. Let us emphasize that we allow a color class in a decomposition to contain isolated vertices. 
A \emph{partial decomposition} of $G$ is a decomposition of some subgraph of~$G$. It is said to be \emph{strict} if it is a decomposition of some proper subgraph of $G$.  For a vertex $v \in V(G)$ and a color class $\G(i)$ of a decomposition $\G$ of $G$, let $d_{\G(i)}(v)$ 
denote the degree of $v$ in $\G(i)$.

There has been a large number of results for problems of the type
``\textit{Given a fixed graph $H$ and a positive integer $n$, find necessary and sufficient conditions for the complete graph $K_{n}$ to be decomposable into color classes each isomorphic to $H$.}"
 For example, the case where $H$ is the complete graph $K_{p}$, $p<n$ corresponds to the theory of Steiner systems and has been solved asymptotically by Wilson~\cite{wilson}. His result has been recently extended to hypergraphs by Keevash~\cite{keevash} in an astonishing proof that settled the long-standing existence conjecture for combinatorial designs. Other results are concerned with decompositions into color classes each isomorphic to a graph satisfying certain properties -- for example one could ask for decompositions into cycles~\cite{bryant:cycles}, paths~\cite{bryant:paths} or stars~\cite{cain:stars}; see~\cite{matthew1, kotzig}  for some further examples.

In this paper, we consider enclosing problems which can be seen as analogues of precoloring problems in the setting of graph decompositions. 

Throughout this section and the rest of the paper, let $\lambda$, $\mu$, $k$,  $r$, $m$, $n$, $p$ be positive integers such that 
\[
\mu \geq \lambda, \quad m \geq n, \quad r \geq 2 \ \text{ and } \;p =  r(2n-m)/2.
\]
Let $\lambda K_{n}$ denote the complete graph on $n$ vertices with multiplicity $\lambda$ (that is, every pair of vertices are joined by $\lambda$ parallel edges).  A decomposition $\G$ into $k$ colors of $\lambda K_n$  is said to be \emph{enclosed} in a decomposition $\H$ into $k$ colors of $\mu K_m$ if, for each $1 \leq i \leq k$, $\G(i)$ is a subgraph of~$\H(i)$. An enclosing problem is  to find conditions on $\G$ to admit an enclosing in an $\H$ of a certain type.

There are a number of enclosing results in the case where $\lambda < \mu$ and the target decomposition consists of color classes each isomorphic to a cycle of one prescribed length~\cite{ Asplund0, asplundone, Asplund2, c31, c32, c33, c42, c4} or each isomorphic to a cycle of one of a number of prescribed lengths~\cite{ceven, cvarying}. However, less is known in the case where $\lambda < \mu$ and the target decomposition consists of spanning subgraphs. The aim of this paper is to address the following enclosing problem in which the target decomposition consists of spanning regular graphs that are sufficiently connected. A decomposition in which each color class is an $r$-regular $t$-edge-connected spanning subgraph is known as a $t$-edge-connected $r$-factorization.

\begin{problem}\label{problem} For $t\geq 2$, find necessary and sufficient conditions for enclosing a given decomposition  of $\lambda K_{n}$ into $k$ colors in some $t$-edge-connected $r$-factorization of $\mu K_{m}$.
\end{problem}

It is not difficult to verify (see Corollary~\ref{lem:necessary}) that in the non-trivial case where $n<m$ or $\lambda<\mu$, a necessary condition on the decomposition of $\lambda K_n$ to answer Problem~\ref{problem} is the following property.
\begin{definition}\label{def:admissible}
A decomposition $\G$ into $k$ colors is \emph{$r$-admissible} if 
\begin{itemize}
 \item $d_{\G(i)}(v) \leq r$ for each $v \in V(K_n)$ and $1\leq i \leq k$,
\item If $C$ is a connected component of $\G(i)$, then $C$ contains a vertex of degree at most $r-2$ in $\G(i)$ or at least two vertices of degree at most $r-1$ in $\G(i)$.
 \item If $C$ is a connected component of $\G(i)$ and $e$ is a cutedge of $C$, then the components $C_1$ and $C_2$ of $C-e$ each contain a vertex of degree at most $r-1$ in $\G(i)$. 
\end{itemize}
\end{definition}

In 1984, Hilton~\cite{hilton1} settled Problem~\ref{problem} where $\lambda = \mu = 1$ and $t = r = 2$, which corresponds to the problem of enclosing a given decomposition of $K_n$ in some Hamiltonian decomposition of $K_m$. 
 Nash-Williams \cite{amalgamation3} proved more general results and gets as a corollary the answer to the case $\lambda=\mu=1$, $t=2$ and $r \geq 2$. We shall find it useful to state this result as formulated and proved by Rodger and Wantland~\cite{RW}. 

\begin{theorem}[\cite{amalgamation3}]\label{thm:rw1}
Suppose that $m>n$, $r\geq 2$, and $p=r(2n-m)/2$. Then a given decomposition $\G$ of $K_n$ into $k$ colors can be enclosed in some $2$-edge-connected $r$-factorization of $K_{m}$
 if and only if 
 \begin{enumerate}[{\normalfont ({\RW}1)}]
 \item $rk=   m-1$ and if $r$ is odd, then $k$ is odd, 
   \item $\G$ is $r$-admissible,
 \item $\min(\{|E(\G(i))|\,:\, 1\leq i\leq k\})\geq p$.
  \end{enumerate}
 \end{theorem}

An investigation of Problem~\ref{problem} in the case where $\lambda<\mu$ and $t = r = 2$ was recently instigated by Feghali and Johnson~\cite{feghali}. They obtained necessary and sufficient conditions for enclosing a given decomposition of $\lambda K_n$ in some Hamiltonian decomposition of $\mu K_{m}$ in several cases: 
\begin{itemize}
\item $m \geq 2n-1$ for any $\mu>\lambda$,
\item $m=4$ and $n=3$ for any $\mu>\lambda$, and
\item $m=2n-2$ for $(\lambda,\mu)=(1,2)$.
\end{itemize}

As one might expect, Problem~\ref{problem} becomes more difficult as $m$ gets smaller or $r$ larger. In this paper, we address Problem~\ref{problem} by first generalizing their result to every~$r \geq 2$ in the following two theorems. The proofs of these theorems  rely on the method of proof in~\cite[Theorem 1.1]{feghali}.  
For a positive integer $i$ and a decomposition $\A$ of a graph $G$, let $\S_i(\A)$ denote the set of color classes of $\A$ that contain \emph{exactly} $i$ edges of $G$, and, for $u,v \in V(G)$, let $\S_i(u,v,\A)$ denote the set of color classes of $\A$ that consist of exactly $i$ edges, all of them between $u$ and $v$.

\begin{theorem}\label{thm:rfactor1}
Suppose that $\mu > \lambda$, $m \geq 2n-1$, $p =  r(2n-m)/2$, and $r\geq 2$. Then a given decomposition $\G$ of $\lambda K_n$ into $k$ colors can be enclosed in some $2$-edge-connected $r$-factorization of $\mu K_m$ if and only if
\begin{enumerate}[{\normalfont({B}1)}]
\item $rk = \mu(m-1)$ and $rm$ is even,
\item $\G$ is $r$-admissible, and
\item $\mathlarger{\sum}_{i=0}^p (p - i)|\S_i(\G)| \leq (\mu - \lambda) \displaystyle{\frac{n(n-1)}{2}}$.
\end{enumerate} 
\end{theorem}

With the additional (and possibly unneeded) assumption $\mu\leq 2r-2$, we can extend Theorem~\ref{thm:rfactor1} to $m=2n-2$. 

\begin{theorem}\label{thm:rfactor2}
Suppose that $2(r-1) \geq \mu > \lambda$, $m = 2n-2$, and $r\geq 2$.  Then a given decomposition $\G$ of $\lambda K_n$ into $k$ colors can be enclosed in some $2$-edge-connected $r$-factorization of $\mu K_m$ if and only~if
\begin{enumerate}[{\normalfont({C}1)}]
\item $rk = \mu(m-1)$ and $rm$ is even,
\item $\G$ is $r$-admissible,
\item $\mathlarger{\sum}_{i=0}^r (r - i)|\S_i(\G)| \leq (\mu - \lambda) \displaystyle{\frac{n(n-1)}{2}}$, and
\item  for each  $u, v \in V(K_n)$, \\ $\displaystyle{|\S_0(\G)| + \sum_{i=1}^{r-1}|\S_i(u, v; \G)| \leq (\mu - \lambda) \left(\frac{n(n-1)}{2} - 1\right)}.$
\end{enumerate}
\end{theorem}

We are also able to show that under the extra assumption that  $\G$ is $(r-1)$-admissible, then $\G$ can always be enclosed in some $2$-edge-connected $r$-factorization of $\mu K_m$ for significantly more values of $m$, $\mu$, and $\lambda$ than those covered by Theorems~\ref{thm:rfactor1} and~\ref{thm:rfactor2}.

\begin{theorem}\label{thm:new}
Suppose that $rk = \mu(m-1)$, $rm$ even, $r\geq 3$ and $2 \mu > r(\mu - \lambda)$. Let $\G$ be an $(r-1)$-admissible decomposition of $\lambda K_n$ into $k$ colors. Then there exists a constant $C = C(\mu, \lambda, r)$ such that if $m \geq (2 - C)n+1$, then $\G$ can be enclosed in some $2$-edge-connected $r$-factorization of $\mu K_m$.
\end{theorem}

The paper is organized as follows. The necessity of (B1)--(B3) and (C1)--(C4) will be established in the next section (see Corollary~\ref{lem:necessary}),
in which we also prove a multigraph version of Theorem \ref{thm:rw1} that will be of use. In the last three sections, we prove Theorems~\ref{thm:rfactor1}, \ref{thm:rfactor2} and~\ref{thm:new}.

\section{Amalgamations and Detachments}\label{section:preliminaries}
We shall use the same strategy to prove each of the three main theorems of this paper. 
In vague terms that are made more precise at the end of this section,  we will first enclose the given decomposition $\G$ of $\lambda K_n$ in a suitable decomposition $\G'$ of   $\mu K_{n}$. Secondly and lastly, we will enclose $\G'$ in some $2$-edge-connected $r$-factorization of $\mu K_m$. This second step will be done through a multigraph analogue of Theorem~\ref{thm:rw1}, that we prove below (it will also yield as a corollary the necessary conditions of our first two theorems).

 \begin{theorem}\label{thm:rw}

Let $r\geq 2$ and $p =  r(2n-m)/2$. A given decomposition $\A$ of $\mu K_n$ into $k$ colors can be enclosed in some $2$-edge-connected $r$-factorization of $\mu K_{m}$
 if and only if 
 \begin{enumerate}[{\normalfont (\RWp1)}]
 \item $rk=\mu (m -1)$ and $rm$ is even,
  \item $\A$ is $r$-admissible, and
 \item $\min(\{|E(\A(i))|\,:\, 1\leq i\leq k\}\geq p$.
  \end{enumerate}
 \end{theorem}
 
Theorem~\ref{thm:rw} will be proved using the technique of \emph{amalgamations}. In particular, it will follow from Theorem 9 in \cite{amalgamation3}.  (In fact, we will use a simplified version of this theorem because we don't need the full generality of the argument in~\cite{amalgamation3}.) In order to state this theorem, we  must introduce the terminology in~\cite{amalgamation3}. 

 In the remainder of this section, we allow our graphs to have loops. For vertices $x$ and $y$ in some graph $H$, we define $d_{H}(x,y)$ (or simply $d(x,y)$ if no confusion is possible) to be the number of edges between $x$ and $y$ in $H$ if $x\neq y$ and the number of loops incident with $x$ if $x=y$. When counting the degree of $x$, the loops count twice. That is, $d(x)=\sum_{y\neq x}d(x,y)+2d(x,x)$.

If $F$ and $G$ are graphs, $\phi$ is a surjection from $V(F)$ onto $V(G)$, $\psi$ a bijection between $E(F)$ and $E(G)$,  such that $e \in E(F)$ joins $x$ and $y$ if and only if $\psi(e) \in E(H)$ joins $\phi(x)$ and $\phi(y)$, we say that $G$ is an {\em amalgamation} of $F$,  $F$ is a {\em detachment} of $G$, and that the functions $\phi$ and $\psi$ are \emph{amalgamation functions}. Informally speaking, each vertex $v$ of $G$ is obtained by identifying all vertices in $F$ which belong to the set $\phi^{-1}(v)$.

 A {\em triad} is a triple $(G,g,\G)$, where $G$ is a graph, $g$ is a function from $V(G)$ into $\mathbb{N}\setminus\{0\}$, such that no vertex $v$ with $g(v)=1$ is incident with a loop and $\G$ is a decomposition of $G$. 
For vertices $v,w\in V(G)$, we define $g(v,w)$ to be $g(v)g(w)$ if $v\neq w$ and ${g(v) \choose 2}$ if $v=w$ (with the interpretation of ${1 \choose 2}$ as $0$ whenever it occurs).

A triad will be called:
\begin{itemize}
\item \emph{expandable} if there exists a vertex $v$ such that $g(v)\geq 2$,
\item \emph{fully expanded} if $g(v) = 1$ for every vertex $v$, and
\item \emph{good} if, for every $i$, $\G({i})$ is 2-connected and $d_{\G({i})}(v)\geq 2 g(v)$ for every vertex $v$.
\end{itemize}
For real numbers $a$ and $b$, $b \approx a$ means $\lfloor a \rfloor \leq b \leq \lceil a \rceil$. Note that $\approx$ is not symmetric. A triad $(F,f,\F=\{\F({1}),\ldots,\F({k})\})$ will be called a {\em fair detachment} of a triad $(G,g,\G=\{\G({1}),\ldots,\G({k})\})$ if $F$ is a detachment of $G$ with amalgamation functions $\phi$ from $V(F)$ onto $V(G)$ and $\psi$ between $E(F)$ and $E(G)$ satisfying
 \begin{enumerate}[{\normalfont ({\NW}1)}]
\item $e\in \F({i})$ if and only if $\psi(e)\in \G({i})$,
\item $g(v)=\sum_{x\in \phi^{-1}(v)} f(x)$ for every vertex $v\in V(G)$,
\item $\displaystyle{\frac{d_{\F({i})}(x)}{f(x)}\approx \frac{d_{\G({i})}(\phi(x))}{g(\phi(x))}}$ for every vertex $x\in V(F)$ and every $i\in\{1,\ldots,k\}$,
\item $\displaystyle{\frac{d_{F}(x,y)}{f(x,y)}\approx \frac{d_{G}(\phi(x),\phi(y))}{g(\phi(x),\phi(y)))}}$ for every pair of distinct vertices $x,y\in V(F)$.
\end{enumerate}

We are now able to state Theorem 9 from \cite{amalgamation3}.

\begin{theorem}[\cite{amalgamation3}]\label{thm:triad}
Every good triad has a fully expanded good fair detachment.
\end{theorem}
We are now ready to prove Theorem \ref{thm:rw}.
\begin{proof}[ Proof of Theorem \ref{thm:rw}]
\textit{Necessity:}  Let $\A$ be a decomposition of $\mu K_{n}$ that can be enclosed in some $2$-edge connected $r$-factorization $\A'$ of $\mu K_{m}$. We prove (\RWp 1). In the target decomposition, every color class is spanning and $r$-regular. Thus, by considering the degrees at a single vertex, we obtain $rk=\mu(m-1)$. Moreover, the sum of degrees in any color class is even, so $rm$ must be even. Thus (\RWp 1) holds. 
 
 We prove (\RWp 2).   The degree condition of $r$-admissibility is obvious as $\A$ is enclosed in an $r$-factorization. To establish the other two admissibility conditions, we consider a component $C$ of some color class of $\A$. Let $F$ be the $2$-edge-connected $r$-factor of $\A'$ that contains $C$. Then ~$|V(F)| > |V(C)|$ since $m > n$. Thus, $C$ cannot only have vertices of degree $r$. Also it is not the case that only one edge can be added between $C$ and $F - C$ since this would create a bridge. This implies the second condition of $r$-admissibility. The third one follows by a similar observation. Thus (\RWp 2) holds.  
  
To prove (\RWp 3), we consider the edges of some color class $\A'(i)$ that are in common with the edges between $V(\mu K_{n})$ and $V(\mu K_{m})\setminus V(\mu K_{n})$. Their number is at most $r(m-n)$ and at least $\sum_{v \in K_n} d_{\A'(i)}(v) - d_{\A(i)}(v)$.  Then, since $d_{\A'(i)}(v) = r$ for each $v \in V(\mu K_m)$, we find that $r(m - n) \geq rn-\sum_{v\in K_{n}} d_{\A(i)}(v)=rn-2|E(\A(i))|$, which implies (\RWp 3).

\textit{Sufficiency:} Let $\A$ be a decomposition of $\mu K_{n}$ satisfying conditions (\RWp 1), (\RWp 2), and (\RWp 3). Let $x_1,x_{2},\ldots, x_{n}$ be the vertices of $ \mu K_{n}$,  let $x_0$ be a new vertex and define a triad $(G,g,\G)$ by:
\begin{itemize}
\item $V(G)=\{x_{0},x_{1},\ldots,x_{n}\}$,
\item $g(x_{i})=1$ for every $i>0$ and $g(x_{0})=m-n$,
\item the edges of $G$ not incident with $x_{0}$ are the same as in $\mu K_n$ and of the same color as in $\mu K_{n}$,
\item for every color $i$, there are $|E(\A({i}))|-p$ loops of color $i$ on $x_{0}$ (this number is an integer thanks to (\RWp  1)), and
\item for every color $i$ and every $j>0$, there are $r-d_{\A({i})}(x_{j})$ edges of color $i$ between $x_{0}$ and $x_{j}$.
\end{itemize}
From this (and using $rk=\mu(m-1)$ from (\RWp 1) for the last two items) one can verify that 
\begin{itemize}
\item for every color $i$ and every $j>0$, $d_{\G({i})}(x_{j})=r$,
\item for every color $i$, $d_{\G({i})}(x_{0})=r(m-n)$,
\item $d_{G}(x_{0},x_{j})=\mu(m-n)$ for every $j>0$, and
\item $d_{G}(x_{0},x_{0})=\frac{1}{2}\mu(m-n)(m-n-1)$.
\end{itemize}
Moreover, it is easy to observe from  (\RWp 2) that every $\G({i})$ is a $2$-edge-connected spanning subgraph of $G$. Hence $(G,g,\G)$ is a good triad and we can apply Theorem \ref{thm:triad} to get a fully expanded good fair detachment $(F,f,\F)$ of $(G,g,\G)$. By definition we deduce the following:
$f(v) = 1$ for each $v \in V(F)$. That is, by (\NW2), $x_0$ has been replaced by $m - n$ vertices in $F$ that we denote $x_{n+1},\dots,x_{m}$. By (\NW1), $\G$ is precisely the restriction of $\F$ to $\{x_{1},\ldots,x_{n}\}$. Every color class of $\F$ is $2$-edge-connected and spanning (since the detachment is good). By (\NW3), $d_{\F({i})}(x_{j})=r$ for every $x_j \in V(F)$.    By (\NW4), $d_{\F}(x,y)=\mu$ for any two distinct vertices $x, y \in V(F)$.
In other words, we have exactly shown the existence of the desired enclosing, which proves our theorem.
\end{proof}

\begin{corollary}\label{lem:necessary}
The conditions in Theorems~{\normalfont\ref{thm:rfactor1}} and {\normalfont\ref{thm:rfactor2}} are necessary. 
\end{corollary}

\begin{proof}
Note that conditions (C1), (C2), and (C3) are the same as (B1), (B2), and (B3). 

Suppose that a decomposition $\G$ of $\lambda K_n$ into $k$ colors can be enclosed in some $2$-edge-connected $r$-factorization $\F$ of $\mu K_m$, and let $\A$ denote the restriction of $\F$ to $\mu K_n$.  
By Theorem~\ref{thm:rw}, $\A$ must satisfy conditions (\RWp 1)-(\RWp3). Note that (\RWp1) is exactly (C1) (or (B1)). Since $\G$ is a subgraph of $\A$ and, by (A$^\prime$2),  $\A$ is $r$-admissible; $\G$ is also $r$-admissible. Thus (C2) (or (B2)) holds. By (A$^\prime$3), each color class of $\A$ contains at least $p$ edges, so we must add, from the edges of $\mu K_n \setminus \lambda K_n$ and for each $ i = 0, \dots, p$, at least $p - i$ edges to each color class of $\S_i(\G)$. Thus (C3) (or (B3)) holds.  

 We are left to prove (C4). So we assume that $m=2n-2$ and consider a color $c\in \{1,\ldots,r\}$ that belongs to $\S_{0}(\G)$ or $\S_{i}(x,y,\G)$ for some $i<r$. This means that, in $\G$, there are exactly $i$ edges with color $c$ all of which are $xy$-edges. On the other hand,  by (\RWp3), $\A$ contains  at least $r$ edges with color $c$, and, by (\RWp2), these edges cannot all be $xy$-edges as this would contradict the second condition of $r$-admissibility. 
So we must add to each color class in $\S_0(\G)$ or $\S(x,y,\G)$ at least one edge with color $c$ that is not an $xy$-edge, which implies (C4).
\end{proof}

Given that the necessity conditions are proved, Theorems \ref{thm:rfactor1} and \ref{thm:rfactor2} will follow if we are able to enclose the decomposition $\G$ of $\lambda K_{n}$ in a decomposition of $\mu K_n$ that meets the necessary and sufficient conditions in Theorem~\ref{thm:rw}. To do that in the next sections, we will proceed each time in two steps: 
\begin{itemize}
\item first, we extend $\G$ by coloring some edges of $\mu K_{n}\setminus \lambda K_{n}$ so that every color class in the resulting decomposition $\G'$ contains $p$ edges and $\G'$ is still $r$-admissible. That is, $\G'$ is a partial decomposition of $\mu K_{n}$ that encloses $\G$ and satisfies (A$^\prime$1), (A$^\prime$2) and (A$^\prime$3).
\item second, we show how to extend $\G'$ (one edge at a time, with sometimes the possibility of recoloring edges in $\G'$ that are not in $\G$) to get a full decomposition of $\mu K_{n}$ that satisfies (A$^\prime$1), (A$^\prime$2) and (A$^\prime$3).
\end{itemize}
In light of the first step, we introduce the following definition.   An $r$-admissible decomposition $\G$ of $\lambda K_n$ is \emph{$p$-extendible with respect to $\mu K_n$} if there exists an $r$-admissible partial decomposition $\G'$ of $\mu K_n$ such that\begin{itemize}
 \item the restriction of $\G'$ to $\lambda K_n$ is $\G$, and
 \item  every color class of $\G'$ contains at least $p$ edges.
 \end{itemize}

\section{Proof of Theorem~\ref{thm:rfactor1}} 

 In this section, we prove Theorem~\ref{thm:rfactor1}.  We require the following two lemmas. 
 
 \begin{lemma}\label{lem:extendible11}
Suppose that $m \geq 2n - 1$, $r\geq 2$, and $p =  r(2n-m)/2$,  and let $\G$ be an $r$-admissible decomposition of $\lambda K_n$. If {\normalfont(B3)} holds, then $\G$ is $p$-extendible with respect to $\mu K_n$.
\end{lemma}

 \begin{proof}
We only need to consider the case $(m, p) = (2n - 1, \frac{r}{2})$ because $p \leq 0$ and the result is trivial when $m \geq 2n$.  Suppose that (B3) holds.
Then we can arbitrarily add $p - i$ edges to each color class of $\S_i(\G)$ for $0 \leq i \leq p$ and obtain a decomposition that will be $r$-admissible.
 \end{proof}
 
  \begin{lemma}\label{lem:2m-1}
Suppose that $m \geq 2n-1$, $\mu > \lambda$, $rk = \mu(m-1)$, and $r\geq 2$. Let $\G$ be a partial $r$-admissible decomposition of $\lambda K_n$ into $k$ colors. Suppose that $\G$ is enclosed in a strict partial $r$-admissible decomposition $\G'$ of $\mu K_n$ into $k$ colors. Then $\G$ can be enclosed in a partial $r$-admissible decomposition of $\mu K_n$ into $k$ colors whose color classes are the same size as those of $\G'$ except for one that contains one more edge.
   \end{lemma}

   \begin{proof}
   Let $e$ be an edge of $\mu K_n \setminus \lambda K_n$ that is not an edge of $\G'$, and let $x$ and $y$ be its incident vertices. To prove the lemma, it suffices to show that we can color $e$ in such a way that the resulting decomposition is $r$-admissible. 
   
 If assigning color $i$ to $e$ gives a decomposition that is not $r$-admissible, then, by considering Definition~\ref{def:admissible}, either
\begin{itemize}
\item $d_{\G'(i)}(x)=r$, or
\item  $d_{\G'(i)}(y)=r$, or
\item $x$ and $y$ belong to the same connected component of $\G'(i)$ and $d_{\G'(i)}(x)=d_{\G'(i)}(y)=r-1$, or
\item $x$ and $y$ belong to the same connected component $C$ of $\G'(i)$, $r - 2= d_{\G'(i)}(x) < d_{\G'(i)}(y) = r - 1$, and every other vertex in $C$ has degree exactly $r$ in $C$.
\end{itemize}
In all possible cases we find that $d_{\G'(i)}(x)+d_{\G'(i)}(y)\geq r$. By summing the degrees of $x$ and $y$ over all $k$ colors, we get a number that is at most the sum of degrees of $x$ and $y$ minus 2 as $e$ is not colored. So $2\mu(n-1)-2\geq rk=\mu(m-1)\geq \mu(2n-2)$ since $m \geq 2n-1$, which is a contradiction. Therefore we can always color an uncolored edge and obtain a decomposition that is $r$-admissible.
   \end{proof}
   
   We can now give the proof of Theorem~\ref{thm:rfactor1}, as outlined in~Section~\ref{section:preliminaries}.

   \begin{proof}[Proof of Theorem~\ref{thm:rfactor1}]
The necessity of (B1), (B2) and (B3) follows from Corollary~\ref{lem:necessary}. To prove sufficiency, it suffices to apply Lemma~\ref{lem:extendible11} and then iteratively apply Lemma~\ref{lem:2m-1} to get a decomposition of $\mu K_n$ that satisfies the conditions of  Theorem~\ref{thm:rw}, which gives us the desired enclosing in $\mu K_{m}$.
   \end{proof}
   
   \section{Proof of Theorem~\ref{thm:rfactor2}}

In this section, we prove Theorem~\ref{thm:rfactor2}. We require the following two lemmas. The first lemma is a generalization of \cite[Proposition 2.3]{feghali}. Its proof is similar to that of~\cite[Proposition 2.3]{feghali} but is considerably shorter. 
 
 \begin{lemma}\label{lem:rextendible}
Suppose that $\mu > \lambda$, $m = 2n - 2$, and $r\geq 2$, and let $\G$ be an $r$-admissible decomposition of $\lambda K_n$. If {\normalfont(C3)} and {\normalfont(C4)} hold, then $\G$ is $r$-extendible with respect to $\mu K_n$.
 \end{lemma}
 
 \begin{proof}
  Suppose that (C3) and (C4) hold. First, from the edges of $\mu K_n \setminus \lambda K_n$, we arbitrarily add  exactly one edge to each color class of $\S_0(\G)$, and let $\G'$ denote the resulting decomposition. A color class is said to be \emph{bad} if it belongs to $\S_i(u, v; \G')$ for some $1 \leq i \leq r - 1$ and $u, v \in V(K_n)$; otherwise it is said to be \emph{good}.   We construct an auxiliary bipartite graph $H$ with bipartition $\{V,W\}$ as follows:
 \begin{itemize}
 \item each vertex in $W$ represents an edge in $\mu K_n\setminus \l K_n$ that is not an edge of $\G'$,
 \item for each $1 \leq i \leq r - 1$ and each good color class in $\S_i(\G')$, add   $r - i$ vertices to $V$ that are each adjacent to every vertex in $W$, and
 \item For each $1 \leq i \leq r - 1$ and each bad color class in $\S_i(u, v; \G')$ for some $u, v \in V(K_n)$, add   $r - i - 1$ vertices to $V$ that are each adjacent to every vertex of $W$ and add one vertex to $V$ that we refer to as a \emph{special $uv$-vertex} that is adjacent to every vertex of $W$ that is not a $uv$-edge. 
 \end{itemize}  
  
  \begin{claim}\label{claim:1}
  $H$ has a matching of size $|V|$. 
  \end{claim}
  
  Before we prove the claim, we show that it implies the lemma. Let $M$ be a matching of size $|V|$. For each edge  in $ M$ that joins a vertex in $V$, corresponding to some color class $A$, to a vertex in $W$, corresponding to some edge $f$, we add $f$ to $A$.  Since $H$ has a matching of size $|V|$, it is not difficult to see that, in the resulting decomposition, a color class either is a color class of $\G$ and hence contains at least $r$ edges, or consists of exactly $r$ edges that do not all join the same pair of vertices due to our special vertices. Thus $\G$ is $r$-extendible with respect to $\mu K_n$. 
  
  Let us now prove the claim.  The claim will follow from Hall's Theorem if we can show that $|N(S)| \geq |S|$ for every subset $S \subseteq V$.   
  Fix $S \subseteq V$. We distinguish two cases. 
  
  \medskip
  
  \noindent
  \textbf{Case 1}: $N(S) = W$. Notice that $|\S_1(\G')| = |\S_0(\G)| + |\S_1(\G)|$ and $\S_i(\G) = \S_i(\G')$ for each $i \geq 2$.  Hence
  \begin{eqnarray*}
 |S| \leq |V| &=& \sum_{i = 0}^r (r-i)|\S_i(\G')| \\ &=& \sum_{i=1}^r (r-i)|\S_i(\G)| + (r-1)|\S_0(\G)| \\ &\leq&  (\mu-\lambda) \frac{n(n-1)}{2} - |\S_0(\G)| = |W| = |N(S)|,
  \end{eqnarray*}
  where the inequality follows from (C3). 

\medskip

\noindent
\textbf{Case 2}: $N(S) \subsetneq W$. This implies that $S$ consists only of special $uv$-vertices for some $u, v \in V(K_n)$. Let $\S_0^{(u, v)}$ denote the set of color classes of $\G'$ that were obtained from $\G$ by adding a $uv$-edge to a color class of $\S_0(\G)$, and let $\overline{\S_0^{(u, v)}}$ denote the set of color classes of $\G'$ that were obtained from $\G$ by adding an edge that is not a $uv$-edge to a color class of $\S_0(\G)$. Hence 
\begin{eqnarray*}
|S| &\leq&  \sum_{i=1}^{r-1}|\S_i(u, v; \G')| \\
&=&  |\S_0^{(u, v)}| +  \sum_{i=1}^{r-1}|\S_i(u, v; \G)| \\
&\leq& (\mu - \lambda) \left(\frac{n(n-1)}{2} - 1\right) - |\overline{\S_0^{(u, v)}}| 
= 
|N(S)|,
\end{eqnarray*}
where the second inequality follows from (C4) and the fact that  $|\S_0^{(u, v)}| + |\overline{\S_0^{(u, v)}}| = |\S_0(\G)|$. This completes the proof of the lemma. 
 \end{proof}
 
 \begin{lemma}\label{lem:2m-2}
Suppose that $m =  2n-2$, $r\geq 2$, $rk = \mu(m-1)$, and $2(r-1) \geq \mu > \lambda$. Let $\G$ be an $r$-admissible decomposition of $\lambda K_n$ into $k$ colors. Suppose that $\G$ is enclosed in a strict partial $r$-admissible decomposition $\G'$ of $\mu K_n$ into $k$ colors.  Then $\G$ can be enclosed in a partial $r$-admissible decomposition of $\mu K_n$ into $k$ colors whose color classes are the same size as those of $\G'$ except for one that contains one more edge.
\end{lemma}

\begin{proof}
Let $e$ be an edge of $\mu K_n \setminus \lambda K_n$ that is not an edge of $\G'$, and let $x$ and $y$ be its incident vertices. To prove the lemma, it suffices to show that we can color $e$ in such a way that the resulting decomposition is $r$-admissible. 

Let $E_{xy}^i$ denote the set of edges with color $i$ that are incident with $x$ or $y$ in $\G'(i)$. If assigning color $i$ to $e$ gives a decomposition that is $r$-admissible, then we do so immediately and are done. If this is not the case, then, by considering Definition~\ref{def:admissible}, we have (as in Lemma \ref{lem:2m-1})

\begin{itemize}
\item[(i)]  $d_{\G'(i)}(x)=r$, or
\item[(ii)]   $d_{\G'(i)}(y)=r$, or
\item[(iii)] $x$ and $y$ belong to the same connected component of $\G'(i)$ and $d_{\G'(i)}(x)=d_{\G'(i)}(y)=r-1$, or
\item[(iv)]   $x$ and $y$ belong to the same connected component $C$ of $\G'(i)$, $r-2 = d_{\G'(i)}(x) < d_{\G'(i)}(y) = r - 1$, and every other vertex in $C$ has degree exactly $r$ in $C$. 
\end{itemize}

Suppose for a contradiction that $E_{xy}^i$ consists of $r - 2$ parallel $xy$-edges and a single edge from $y$ to some vertex $w$ distinct from $x$. If all the vertices distinct from $x$ and $y$ in the component $C$ of $\G'(i)$ containing $y$ have degree exactly $r$ in $\G'(i)$, then $yw$ is a cutedge of $C$ such that $C - yw$ contains a component whose vertices each have degree $r$ in $C$. This contradicts Definition \ref{def:admissible}. Combined with (i)--(iv), this implies that either
$|E_{xy}^i| \geq r$ or $E_{xy}^i$ consists of $r - 1$ parallel $xy$-edges. Moreover, because $\sum_{i=1}^k|E_{xy}^i|$ is at most the number of edges incident with $x$ or $y$ minus $1$ as $e$ is not colored,
\[\sum_{i=1}^k|E_{xy}^i| \leq 2\mu(n-2)+(\mu-1) = \mu(m - 1) - 1 = rk - 1,\]
and, because $2(r-1) \geq \mu$ and $e$ is not colored, at most one color class contains at least $r - 1$ parallel $xy$-edges. Combining these facts, there must exist some $j \in \{1, \dots, k\}$ such that $E_{xy}^j$ consists of $r - 1$ parallel $xy$-edges and $|E_{xy}^i| = r$ for each $i \in \{1, \dots, k\} \setminus \{j\}$. Moreover, since $\G'$ is $r$-admissible, $\G'(j)$ contains at least two components $Q_1$, $Q_2$ where $Q_1$ contains all of the edges in $E_{xy}^j$ and $Q_2$ contains a vertex $u$ of degree strictly less than $r-1$ or exactly two vertices $u$ and $v$ of degree exactly $r-1$ (possibly, $Q_2$ consists of a single vertex).  

Let $f$ be an $xu$-edge of $\mu K_n \setminus \lambda K_n$. If $f$ is not an edge of $\G'$, then we can clearly assign color $j$ to~$f$.  
So we can assume that $f$ has some color $c \not = j$. Let us argue that recoloring $f$ to $j$ and then coloring $e$ with $c$ gives us a decomposition $\G''$ that is $r$-admissible. 

Notice that the only cause of difficulty is verifying the conditions of $r$-admissibility in the resulting color class $\G''(c)$. Since $c\neq j$, $|E_{xy}^c|=r$ by the above discussion, and since $f$ is incident to $x$ but not $y$,  $d_{\G'(c)}(y)\leq r-1$. So in $\G''$ no vertex will have degree more than $r$ in color class $\G''(c)$ and hence $\G''(c)$ satisfies the first $r$-admissibility condition. 

To prove that $\G''(c)$ satisfies the other two conditions of $r$-admissibility, we must only show that the component of $\G''(c)$ containing $x$ and $y$ satisfies these conditions. Indeed, every other component of $\G''(c)$ is either a component of $\G'(c)$ and hence is $r$-admissible or it contains vertex $u$ in case $f$ was a cutedge of $\G'(c)$ and hence, using the fact that $u$ has degree at most $r-1$ in that component,  is readily seen to be $r$-admissible. So, from now on, we focus on the component of $\G''(c)$ containing $x$ and $y$.

Since $|E_{xy}^c|=r$, if $E_{xy}^c$ contains at most $r- 3$ parallel $xy$-edges, then $\G''$ clearly satisfies the second condition of $r$-admissibility since the sum of degrees in $\G''(c)$ of $x$ and $y$ will not exceed $2r-2$. It is also readily seen that the third condition of $r$-admissibility is satisfied in this case. Thus, since there are at most $r-2$ parallel edges of color $c$ between $x$ and $y$, the only case remaining is when $E^c_{xy}$ consists of exactly $r - 2$ parallel $xy$-edges, the edge $f$ and another edge $f'$ incident with $x$ or $y$ but not both. 

Irrespective of whether $f'$ is incident with $x$ or $y$, $f'$ is a cutedge in $\G''(c)$ separating a component containing $x$ and $y$ (where one has degree $r$ and the other $r-1$ in $\G''(c)$) from another component that we denote $C$. If $f'$ was already a cutedge in $\G'(c)$, then the degrees of vertices in $C$ are the same in $\G''(c)$ and $\G'(c)$ and hence, by the third condition of $r$-admissibility, $C$ contains a vertex of degree at most $r-1$ in $\G''(c)$. Since one of $x$ and $y$ has degree $r-1$ in $\G''(c)$,  it follows that $\G'(c)$ is $r$-admissible.  If, on the other hand, $f'$ was not a cutedge in $\G'(c)$,  then $u$ belongs to the component of $\G''(c)$ containing $x$ and $y$. So the second condition of $r$-admissibility is satisfied because the sum of degrees in each component did not change from $\G'(c)$ to $\G''(c)$. The third condition of $r$-admissibility is also satisfied because $u$ now has degree $r-1$ in $\G''(c)$. This completes the proof. 
\end{proof}

   \begin{proof}[Proof of Theorem~\ref{thm:rfactor2}]
   The proof is analogous to that of Theorem~\ref{thm:rfactor1} with an application of Lemmas~\ref{lem:rextendible} and~\ref{lem:2m-2} in place of Lemmas~\ref{lem:extendible11} and~\ref{lem:2m-1}. 
   \end{proof}

\section{Proof of Theorem~\ref{thm:new}}

 In this section, we prove Theorem~\ref{thm:new}. First, we require an easy observation.

 \begin{lemma}\label{lem:pq}
 Suppose that  $r \geq 3$, $r\geq 2$, and $\mu > \lambda$, and  let $\G$ be an $(r-1)$-admissible decomposition of $\lambda K_n$ into $k$ colors. If $\B$ is a proper $k$-edge-coloring of $\mu K_n \setminus \lambda K_n$,  then $ \G \cup \B$ is an $r$-admissible decomposition of~$\mu K_n$.   
 \end{lemma}
 
 \begin{proof}
We must prove that $\C = \G \cup \B$ satisfies the conditions of Definition~\ref{def:admissible}. Note that $d_{\C(i)}(v) \leq r$ for each $v \in V(K_n)$ and $1 \leq i \leq k$. Note also that every component of every color class of $\C$ contains a vertex of degree at most $r  - 2$ or at least two vertices of degree $r - 1$. To see this, suppose for a contradiction that some component $C$ of some color class $\C(i)$ of $\C$ satisfies $\sum_{v \in V(C)} d_{\C(i)}(v) \geq |V(C)|r - 1$. Then there must be a component of $\G(i)$ that is a subgraph of $C$ which is either $(r - 1)$-regular or has all but one
vertex of degree $r - 1$ with the other of degree $r - 2$, contrary to our assumption that $\G$ is $(r-1)$-admissible. 
 
 Finally, suppose that $e$ is a cutedge of $C$, and let $C_1$ and $C_2$ be the components of $C - e$. We must show that $C_1$ and $C_2$ each have a vertex of degree at most $r - 1$ in $\C(i)$. Suppose for a contradiction that every vertex of $C_1$ or $C_2$, say $C_1$, has degree $r$ in $\C(i)$. Let $A_1$ denote the restriction of $\G(i)$ to $C_1$. Notice that each vertex of $A_1$ has degree precisely $r-1$ in $\G(i)$. Hence $A_1$ is connected and $e$ is a member of $\G(i)$ since otherwise $\G(i)$ contains a component whose vertices each have degree $r-1$, contradicting that $\G$ is $(r-1)$-admissible. It follows that $e$ is a cutedge of some component $A$ of $\G(i)$ such that $A - e$ contains $A_1$, again contradicting that $\G$ is $(r-1)$-admissible. 
 \end{proof}
 
We also need the following specialisation of Corollary 3 from \cite{bryant} due to Bryant. A decomposition $\G$ of a graph $G$ into $k$ colors is said to be \emph{almost regular} if $|d_{\G(i)}(x) - d_{\G(i)(y)}| \leq 1$ for every $x, y \in V(G)$ and $i = 1, \dots, k$.

 \begin{lemma}[\cite{bryant}]\label{lem:complete}\label{lem:bryant}
Let $n$, $\lambda$, $t$ be positive integers, and let $r_{i}$ be a non-negative integer for $i = 1, \dots, t$. Then there is an almost-regular decomposition of $\lambda K_n$ with $t$ colors $c_1,c_2,\dots,c_t$ such that each color $c_i$ has exactly $r_i$ edges if and only if $\sum_{j=1}^t r_j \leq \lambda \binom{n}{2}$.
 \end{lemma}

\begin{proof}[Proof of Theorem \ref{thm:new}]
By Theorem \ref{thm:rw}, it suffices to construct a decomposition of $\mu K_{n}$ that encloses $\G$ that is $r$-admissible 
and contains at least $p$ edges per color class. 

From Lemma \ref{lem:bryant}, there is a decomposition $\F$ of $\mu K_n \setminus \lambda K_n$ into $k$ colors such that, for each $i = 1, \dots, k$, $|E(\F_i)| \in \{\lfloor \frac{\mu - \lambda}{k}\binom{n}{2} \rfloor, \lceil \frac{\mu - \lambda}{k}\binom{n}{2}\rceil \} $ and $|d_{\F(i)}(x) - d_{\F(i)}(y)| \leq 1$ for all $x, y \in V(K_n)$. We will show that for an appropriate choice of $C$, $\F$ has the required conditions. 

If we assume that $C < 2 - \frac{r(\mu - \lambda)}{\mu}$, then $k = \frac{\mu(m - 1)}{r} \geq \frac{\mu(2 - C)n}{r} \geq (\mu - \lambda)n$. This implies that $\frac{\mu - \lambda }{k}\binom{n}{2} \leq \frac{n - 1}{2}$ and hence that $\F$ is a proper edge coloring. By Lemma \ref{lem:pq}, $\G \cup \F$ is $r$-admissible so, since $p$ is an integer, it remains to show that $p \leq \frac{\mu - \lambda}{k}\binom{n}{2}$. Note that we can assume that $m \leq 2n - 1$ since $p \leq 0$ whenever $m \geq 2n$. Now, if we further assume $C \leq \frac{\mu - \lambda}{2\mu}$, we have the required result because 
\[
p = \frac{r(2n - m)}{2} \leq \frac{Cnr}{2} \leq \frac{(\mu - \lambda)nr}{4\mu} = \frac{(\mu - \lambda)n(m - 1)}{4k} \leq \frac{(\mu - \lambda)}{k} \binom{n}{2}
\]
So by taking $C \leq \min \{\frac{\mu - \lambda}{2\mu},  2 - \frac{r(\mu - \lambda)}{\mu}\}$, the theorem follows. 	  
\end{proof}

\section{Concluding remarks}

 Feghali and Johnson~\cite{feghali} gave an example of a decomposition of $5 K_4$ that satisfies conditions (C1)--(C4) with $r = 2$ but that cannot be enclosed in some Hamiltonian decomposition of $6 K_6$. However, we think that if $n$ sufficiently large and $m \geq 2n - 2$,  then (C1)--(C4) are likely to always be sufficient conditions.  (Observe that (C3) and (C4) always hold whenever $n$ is large compared to $r$.)

\begin{conjecture}
Let $n$ and $m$ be positive integers such that $m = 2n - 2$ and $r\geq 2$. Suppose that $r$, $\mu$, and $\lambda$ are positive integers that do not depend on $n$ such that $\mu > \lambda$. Then a given decomposition $\G$ of $\lambda K_n$ into $k$ colors can be enclosed in some $2$-edge-connected $r$-factorization of $\mu K_m$ if and only if $rk = \mu(m - 1)$, $rm$ is even and $\G$ is $r$-admissible, provided $n$ is sufficiently large.  
\end{conjecture}  

 In the situation where $m < 2n - 2$, Theorem~\ref{thm:new} provides extensive solutions under the assumptions that the given decomposition is $(r-1)$-admissible and $2 \mu > r(\mu - \lambda)$. Theorem~\ref{thm:new} depends on more conditions than our first two theorems (but covers significantly more values of $m$) so we only state the following rather general problem and do not attempt to give a conjecture. 
 
 \begin{problem}
 Let $r$, $\mu$, $\lambda$, $n$ and $m$ be positive integers such that $\mu > \lambda$ and $m \geq (2 - C)n$. Is it true that we may choose $C = C( \lambda, \mu, r)$ so that a decomposition $\G$ of $\lambda K_n$ into $k$ colors can always be enclosed in some $2$-edge-connected $r$-factorization of $\mu K_m$ whenever $rk = \mu(m - 1)$, $rm$ is even and $\G$ is $r$-admissible, provided $n$ is sufficiently large? 
 \end{problem} 
 
 \section*{Acknowledgments} 
The authors are indebted to both referees for several helpful suggestions and to one referee for supplying the current proof of Theorem \ref{thm:new}.   This paper received partial support by Fondation Sciences Math\'ematiques de Paris and the Research Council of Norway via the project CLASSIS, Grant Number 249994.

 \bibliography{bibliography}{}

\begin{thebibliography}{10}

\bibitem{Asplund0}
J.~Asplund.
\newblock 5-cycle systems of {$(\lambda+ m) K_{v+ 1} - \lambda K_v$ and
  $\lambda K_{v+ u}- \lambda K_v$}.
\newblock {\em Discrete Mathematics}, 338(5):766--783, 2015.

\bibitem{ceven}
J.~Asplund.
\newblock {$m$}-cycle packings of {$(\lambda+ \mu) K_{v+ u}-\lambda K_v :$}
  {$m$} even.
\newblock {\em arXiv}, 1511.09301, 2015.

\bibitem{asplundone}
J.~Asplund, C.~A. Rodger, and M.~S. Keranen.
\newblock Enclosings of {$\lambda$}-fold {$5$}-cycle systems: Adding one
  vertex.
\newblock {\em Journal of Combinatorial Designs}, 22(5):196--215, 2014.

\bibitem{Asplund2}
J.~Asplund, C.~A. Rodger, and M.~S. Keranen.
\newblock Enclosings of {$\lambda$}-fold 5-cycle systems for {$u=2$}.
\newblock {\em Discrete Mathematics}, 338(5):743 -- 765, 2015.

\bibitem{bryant:paths}
D.~Bryant.
\newblock Packing paths in complete graphs.
\newblock {\em Journal of Combinatorial Theory, Series B}, 100(2):206--215,
  2010.

\bibitem{bryant}
D.~Bryant.
\newblock On almost-regular edge colourings of hypergraphs.
\newblock {\em The Electronic Journal of Combinatorics}, 23(4):4--7, 2016.

\bibitem{bryant:cycles}
D.~Bryant, D.~Horsley, and W.~Pettersson.
\newblock Cycle decompositions v: Complete graphs into cycles of arbitrary
  lengths.
\newblock {\em Proceedings of the London Mathematical Society}, page pdt051,
  2013.

\bibitem{cain:stars}
P.~Cain.
\newblock Decomposition of complete graphs into stars.
\newblock {\em Bulletin of the Australian Mathematical Society}, 10(01):23--30,
  1974.

\bibitem{c31}
C.~J. Colbourn, R.~C. Hamm, and A.~Rosa.
\newblock Embedding, immersing, and enclosing.
\newblock {\em Congressus Numerantium}, 47:229--236, 1985.

\bibitem{feghali}
C.~Feghali and M.~Johnson.
\newblock Enclosings of decompositions of complete multigraphs in
  2-factorizations.
\newblock {\em Journal of Combinatorial Designs}, to appear.

\bibitem{matthew1}
A.~J. Hilton and M.~Johnson.
\newblock An algorithm for finding factorizations of complete graphs.
\newblock {\em Journal of Graph Theory}, 43(2):132--136, 2003.

\bibitem{hilton1}
A.~J.~W. Hilton.
\newblock Hamiltonian decompositions of complete graphs.
\newblock {\em Journal of Combinatorial Theory, Series B}, 36(2):125--134,
  1984.

\bibitem{cvarying}
D.~Horsley and R.~A. Hoyte.
\newblock Decomposing ${K}_{u+w}-{K}_u$ into cycles of prescribed lengths.
\newblock {\em Discrete Mathematics}, 340(8):1818 -- 1843, 2017.

\bibitem{c32}
S.~P. Hurd, P.~Munson, and D.~G. Sarvate.
\newblock Minimal enclosings of triple systems {I}: adding one point.
\newblock {\em Ars Combinatoria}, 68:145--159, 2003.

\bibitem{c33}
S.~P. Hurd and D.~G. Sarvate.
\newblock Minimal enclosings of triple systems {II}: increasing the index by 1.
\newblock {\em Ars Combinatoria}, 68:263--282, 2003.

\bibitem{keevash}
P.~Keevash.
\newblock The existence of designs.
\newblock {\em arXiv preprint arXiv:1401.3665}, 2016.

\bibitem{kotzig}
A.~Kotzig.
\newblock Decompositions of complete graphs into isomorphic cubes.
\newblock {\em Journal of Combinatorial Theory, Series B}, 31(3):292--296,
  1981.

\bibitem{amalgamation3}
C.~{\relax St}. J.~A. Nash-Williams.
\newblock Amalgamations of almost regular edge-colourings of simple graphs.
\newblock {\em Journal of Combinatorial Theory, Series B}, 43(3):322--342,
  1987.

\bibitem{c42}
N.~A. Newman.
\newblock 4-cycle decompositions of {$(\lambda+ m) K_{v+ u} \setminus \lambda
  K_v$}.
\newblock {\em Designs, Codes and Cryptography}, 75(2):223--235, 2015.

\bibitem{c4}
N.~A. Newman and C.~A. Rodger.
\newblock Enclosings of {$\lambda$}-fold 4-cycle systems.
\newblock {\em Designs, Codes and Cryptography}, 55(2-3):297--310, 2010.

\bibitem{RW}
C.~Rodger and E.~B. Wantland.
\newblock Embedding edge-colorings into 2-edge-connected $k$-factorizations of
  ${K}_{kn+ 1}$.
\newblock {\em Journal of Graph Theory}, 19(2):169--185, 1995.

\bibitem{wilson}
R.~M. Wilson.
\newblock The necessary conditions for t-designs are sufficient for something.
\newblock {\em Utilitas Mathematica}, 4.

\end{thebibliography}
\bibliographystyle{abbrv}
 
\end{document}